\documentclass[12pt]{amsart}
\usepackage{amssymb}
\input amssym.def
\usepackage{amsmath,amsfonts,hyperref,xcolor,mathtools}
\usepackage{amscd}
\usepackage[mathscr]{eucal}
\usepackage{enumitem}
\usepackage[thinc]{esdiff}
\usepackage{orcidlink}

\hypersetup{colorlinks=true,citecolor={purple},linkcolor={teal},urlcolor={violet}}

\newcommand{\Q}{{\mathbb Q}}
\newcommand{\C}{{\mathbb C}}
\newcommand{\R}{{\mathbb R}}

\newtheorem{thm}{Theorem}
\newtheorem{lem}{Lemma}
\newtheorem{cor}{Corollary}

\newtheorem{rmk}{Remark}

\newcommand{\thmref}[1]{Theorem~\ref{#1}}

\newcommand{\lemref}[1]{Lemma~\ref{#1}}
\newcommand{\corref}[1]{Corollary~\ref{#1}}
\newcommand{\rmkref}[1]{Remark~\ref{#1}}

\makeatletter
\@namedef{subjclassname@2020}{%
\textup{2020} Mathematics Subject Classification}
\makeatother

\begin{document}

%\date{\today}

\title[Ramaswami Type translation formulae for the polylogarithm functions]
{Ramaswami Type translation formulae for the polylogarithm functions}

\author{Pawan Singh Mehta \orcidlink{0009-0001-6717-9063} and Biswajyoti Saha\orcidlink{0009-0009-2904-4860}}

\address{Pawan Singh Mehta and Biswajyoti Saha\\ \newline
Department of Mathematics, Indian Institute of Technology Delhi, 
Hauz Khas, New Delhi 110016, India.}
\email{maz218521@maths.iitd.ac.in, biswajyoti@maths.iitd.ac.in}



\begin{abstract}
In 1934, Ramaswami proved a number of curious translation formulae satisfied by the
Riemann zeta function. Such translation formulae, in turn give the meromorphic extension
of the Riemann zeta function. In 1954, Apostol extended those identities to establish a
family of such similar translation formulae. In this article, we establish many such
Ramaswami and Apostol type translation formulae for the Dirichlet series defining the
polylogarithm functions. This extended set up has many interesting applications, for example,
it allows us to also find some (seemingly new)
recurrence relations between the Bernoulli numbers, and use them to deduce some
congruence properties of the tangent numbers.
\end{abstract}

\maketitle

\section{Introduction}
For a complex number $s$ with $\Re(s)>1$, the Riemann zeta function $\zeta(s)$ is defined by
$$
\zeta(s):=\sum_{n\ge 1}\frac{1}{n^s}.
$$
It defines a holomorphic function in the complex half plane $\Re(s)>1$ and can be
continued to a meromorphic function in $\C$ having a simple pole at $s=1$ with residue 1. 
V. Ramaswami \cite{VR} obtained some interesting translation formulae involving 
the Riemann zeta function, which can be used recursively to establish the analytic 
continuation of $\zeta(s)$.
For each $n\geq 0$, let $P_n(s)$ be defined by 
$$
P_0(s)=1, P_n(s)=\frac{s(s+1)\cdots(s+n-1)}{n!}.
$$
Then for complex number $s$ with $\Re(s)>1$, he \cite{VR} proved the following identities:
\begin{equation}\label{Ramaswami-2}
\zeta(s)\left(1-2^{1-s}\right)=\sum_{n\ge 1}\frac{P_n(s)\zeta(s+n)}{2^{s+n}},
\end{equation}
\begin{equation}\label{Ramaswami-3}
\zeta(s)\left(1-3^{1-s}\right)=1+2\sum_{n\ge 1}\frac{P_{2n}(s)\zeta(s+2n)}{3^{s+2n}},
\end{equation}
\begin{equation}\label{Ramaswami-6}
\zeta(s)\left(1-2^{1-s}-3^{1-s}-6^{1-s}\right)=1+2\sum_{n\ge 1}\frac{P_{2n}(s)\zeta(s+2n)}{6^{s+2n}}.
\end{equation}
He then used the last two formulae to show non-vanishing of $\zeta(s)$ in the region $1/2 \le \Re(s) \le 1$ and $|s| \le 8$.
These similar looking formulae however does not seem to have a common natural extension.
T. Apostol \cite{TA} proved two different formulae, one extending \eqref{Ramaswami-2}, and another
extending \eqref{Ramaswami-3} and \eqref{Ramaswami-6}.
For any given integer $k\geq 2$  and a complex number $s$ with $\Re(s)>1$, he proved:
\begin{equation}\label{Apostol-trans-1}
\zeta(s)\left(1-k^{1-s}\right)=\sum_{n\ge 1}\frac{P_{n}(s)\zeta(s+n)}{k^{s+n}}\frac{B_{n+1}(k)-B_{n+1}}{n+1},
\end{equation}
where for $n,k\geq 0$, $B_{n}(k)$ are the values of Bernoulli polynomial $B_n(x)$ defined by the generating function
\begin{align}\label{Bernoulli-poly}
\frac{te^{xt}}{e^t-1}=\sum_{n\geq 0}B_{n}(x)\frac{t^n}{n!}
\end{align}
and $B_n:=B_n(0)$ is the $n$th Bernoulli number.
Note that when $k=2$, \eqref{Apostol-trans-1} is same as \eqref{Ramaswami-2} (using \eqref{m-power-sum} below).
Moreover, for all integers $k>2$, Apostol \cite{TA} proved that
\begin{align}\label{TA-formula}
\zeta(s)\sum_{d|k}\mu(d)d^{-s}=\phi(-s,k)+2\sum_{n\ge 0}\frac{P_{2n}(s)\zeta(s+2n)}{k^{s+2n}}\phi(2n,k),
\end{align}
where $\mu(n)$ is the Möbius function and  $\phi(s,k):=\sum_{\substack{1\leq h\leq \frac{k}{2}\\ (h,k)=1}}h^{s}$. 
Taking $k=3, 6$ in \eqref{TA-formula}, one gets back \eqref{Ramaswami-3}, \eqref{Ramaswami-6}, respectively.

 In this article, we consider the polylogarithm function $\mathrm{Li}_z(s)$, where $z$ is a root of unity,
 defined for complex numbers $s$ with $\Re(s)>1$ as a Dirichlet series by
\begin{equation}
\mathrm{Li}_z(s):=\sum_{n\geq1}\frac{z^n}{n^s}.
\end{equation}
The above series converges uniformly in every compact subset of the complex half plane 
$\Re(s)>1$. Therefore, it defines a holomorphic function in the complex half plane $\Re(s)>1$. If $z=1$ then 
$\mathrm{Li}_z(s)=\zeta(s)$. Hence it has a simple pole at $s=1$ with residue 1 and for other values of $z$,
$\mathrm{Li}_z(s)$ can be extended to an entire function. 
The aim of this article is to obtain Ramaswami and Apostol type formulae involving $\mathrm{Li}_z(s)$.
Another translation formula satisfied by $\mathrm{Li}_z(s)$, which can also be used to get the analytic continuation
of $\mathrm{Li}_z(s)$, can be found in \cite{BS}.

We also exhibit some applications of such identities. 
For example, we obtain some (seemingly new) recurrence formulae for the Bernoulli numbers.
We show that (see \eqref{Bernoulli-recurrence-3} below) for every integer $n\geq 1$, 
the even indexed Bernoulli number $B_{2n}$ satisfies the recurrence relation
$$
(1-2^{2n})(1-3^{2n-1})\frac{B_{2n}}{2n}= \frac{2^{2n-1}-1}{2} +\sum_{i=1}^{n-1}\binom{2n-1}{2i-1}\frac{B_{2i}}{2i} \ 3^{2i-1}(1-2^{2i})(2^{2n-2i}-1).
$$
This has in turn interesting applications regarding the tangent numbers, defined for $n \ge 1$ by
$$
T_n:=(-1)^{n}2^{2n}(1-2^{2n})\frac{B_{2n}}{2n}.
$$
Note that the $n$-th tangent number $T_n$ is the coefficient of $x^{2n-1}/(2n-1)!$ in the Taylor expansion
of the tangent function around $x=0$, i.e.,
$$
\tan x= \sum_{n \ge 1} T_n \frac{x^{2n-1}}{(2n-1)!}.
$$
Properties of the Tangent numbers are of classical interest (see \cite{KB}).
It is well-known that $T_n$'s are integers \cite{KB} (also see \cite[Remark 1.18]{AIK}).
From the work of \cite{HL}, it also follows that they are all even integers for $n \ge 2$ (also see \cite{ZLZ}).

In this article, as an application of our formulae, we prove that for $n \ge 2$, the last digit of $T_n$ is either $2$ or $6$,
depending on $n$ is even or odd. We also prove that the Tangent numbers are never divisible by $3$ and give
explicit congruence relations modulo $3$.

 We also get some series expressions involving the odd zeta values, for example,
 \begin{equation}\label{zeta-series}
-\frac{3}{16}=\sum_{m\ge 1}\frac{m}{3^{2m+1}}\zeta(2m+1)(2^{2m-1}-1)(2^{-2m}-1).
\end{equation}
This brings us to the notion of {\it $\zeta$-series representation} for a real number
(see \cite{BBC} for an analogous concept). We say that a real number $x$ has a
{\it $\zeta$-series representation} if there exists a sequence of rational numbers $(p_n/q_n)_{n \ge 2}$ 
 such that $x$ can be written as a convergent infinite series of the form
 $$
\sum_{n\geq 2}\frac{p_n}{q_n}\zeta(n).
 $$
Using \eqref{zeta-series}, we thus get that every rational number has a $\zeta$-series representation. Moreover, we get
that every rational number has a $\zeta$-series representation, which only includes the odd zeta values. This can also
be concluded from \cite[eq. (2)]{VR}. 

\section{Translation Formulae for $\mathrm{Li}_{z}(s)$} 
\label{Sec:Trans-formula}

In this section, we present Ramaswami and Apostol type translation formulae for $\mathrm{Li}_z(s)$ where $z$ is a primitive $q$-th root of unity
for $q\geq1$. Define $\delta_q$ to be the integral part of $1/q$. We prove the following:

\begin{thm}\label{trans-thm}
Let $s\in \C$ with $\Re(s)>1$ and $z$ be a primitive $q$-th root of unity with $q\geq1$. Let $k\geq 2$
be an integer of the form $bq+1$ for some positive integer $b$. Then we have,
\begin{align}\label{trans-formula-1}
\mathrm{Li}_z(s)(1-k^{\delta_{q}-s})=\sum_{m\geq1}\frac{P_m(s)\mathrm{Li}_z(s+m)}{k^{s+m}} \sum_{h=1}^{k-1} z^{-h}h^m.
\end{align}
\end{thm}

For $z=1$, \eqref{trans-formula-1} reduces to \eqref{Apostol-trans-1}, using the formula (see \cite[p. 58]{AIK})
\begin{equation}\label{m-power-sum}
\sum_{h=1}^{k-1}h^m=\frac{B_{m+1}(k)-B_{m+1}}{m+1},
\end{equation}
which is true for any positive integers $k,m\geq 1$.
Hence \eqref{trans-formula-1} can be deemed as an extension of Apostol's translation formula \eqref{Apostol-trans-1}. 
We also establish a closed formula for the sum $\sum_{h=1}^{k-1}z^hh^m$, similar to \eqref{m-power-sum},
when $z\neq1$ is a primitive root of unity.
As it turns out, this sum is related to the values of the generalised Euler polynomials defined in \cite[Appendix A]{PB}.
Using this formula, we obtain another form of the translation formula \eqref{trans-formula-1} for $\mathrm{Li}_z(s)$,
when $z\neq 1$ is a root of unity (see \eqref{Apostol-type-trans-1} below).

We first recall the definition of the generalised Euler polynomials. For a given integer $q\geq2$, the $n$-th generalised
Euler polynomial $E_{q,n}(x)$  is defined by the following generating series:
\begin{align}\label{Euler-poly}
\frac{qe^{xt}}{1+e^t+\cdots+e^{(q-1)t}}=\sum_{n\ge 0}E_{q,n}(x)\frac{t^n}{n!}.
\end{align}
 
 \begin{cor}\label{cor-Apostol-type-trans-1}
Let $q\geq2$ be an integer and other notations be as in Theorem \ref{trans-thm}. Then we have,
\begin{align}\label{Apostol-type-trans-1}
\mathrm{Li}_z(s)(1-k^{-s})=\frac{1}{q}\sum_{m\geq1}\frac{P_m(s)\mathrm{Li}_z(s+m)}{k^{s+m}}\sum_{j=1}^{q-1}\frac{z^{j}-1}{z^j(z-1)}\left(E_{q,m}(j)-E_{q,m}(k+j-1)\right).
\end{align}
\end{cor}

\begin{rmk}\label{rmk-trans-mobius}\rm
With the notations as in Theorem \ref{trans-thm}, we have the following alternating version of the translation formula \eqref{trans-formula-1}:
\begin{align}\label{rmk-trans-mobius-alter}
\mathrm{Li}_z(s)(1-k^{\delta_q-s})=\sum_{h=1}^{k-1}z^hh^{-s}
+\sum_{m\geq1}\frac{(-1)^mP_m(s)\mathrm{Li}_z(s+m)}{k^{s+m}}\sum_{h=1}^{k-1}z^{h}h^m.
\end{align}
\end{rmk}

The identities \eqref{trans-formula-1} and \eqref{rmk-trans-mobius-alter} lead to the following two identities, by addition and subtraction:
\begin{align}\label{trans-add}
 2\ \mathrm{Li}_z(s)(1-k^{\delta_q-s})=\sum_{h=1}^{k-1}z^hh^{-s}
 +\sum_{m\geq1}\frac{P_m(s)\mathrm{Li}_z(s+m)}{k^{s+m}}\sum_{h=1}^{k-1}\left(z^{h}(-1)^m+z^{-h}\right)h^m,
\end{align}
\begin{align}\label{trans-sub}
 0=\sum_{h=1}^{k-1} z^hh^{-s}
 +\sum_{m\geq1}\frac{P_m(s)\mathrm{Li}_z(s+m)}{k^{s+m}}\sum_{h=1}^{k-1}\left(z^{h}(-1)^m-z^{-h}\right)h^m.
\end{align}
Now if we substitute $z=1$ and $k=2$ in \eqref{trans-add} and \eqref{trans-sub},
we recover \cite[eq. (1) and (2)]{VR}. Ramaswami used these identities to give $\zeta$-series representations of
$1$, $\log 2$, Euler's constant $\gamma$, and $\zeta(3)/\zeta(2)$ (see \cite[\S 3, \S 5]{VR}).
We further derive the following formulae, involving the Möbius function $\mu(n)$.

\begin{thm}\label{thm-trans-mobius}
With the notations as in Theorem \ref{trans-thm}, we have,
\begin{align}\label{trans-mobius-alter}
 \sum_{d|k}\frac{\mu(d)}{d^s}\mathrm{Li}_{z^d}(s)=\sum_{ \substack{1\leq h\leq k \\(h,k)=1}}z^hh^{-s}
 +\sum_{m\geq0}\frac{(-1)^mP_m(s)\mathrm{Li}_z(s+m)}{k^{s+m}}\sum_{\substack{1\leq h\leq k \\(h,k)=1}}z^{h}h^m,
\end{align}
and
\begin{align}\label{trans-mobius}
 \sum_{d|k}\frac{\mu(d)}{d^s}\mathrm{Li}_{z^d}(s)=\sum_{m\geq0}\frac{P_m(s)\mathrm{Li}_z(s+m)}{k^{s+m}}\sum_{\substack{1\leq h\leq k \\(h,k)=1}}z^{-h}h^m.
 \end{align}
\end{thm}

Subtracting \eqref{trans-mobius} from \eqref{trans-mobius-alter}, we get the following translation formula:
\begin{align}\label{trans-mobius-sub}
 0=\sum_{ \substack{1\leq h\leq k \\(h,k)=1}}z^hh^{-s}
 +\sum_{m\geq0}\frac{P_m(s)\mathrm{Li}_z(s+m)}{k^{s+m}}\sum_{\substack{1\leq h\leq k \\(h,k)=1}}\left(z^{h}(-1)^m-z^{-h}\right)h^m.
\end{align}
Note that if we substitute $z=1$ and $k=3$ in \eqref{trans-mobius}, we get
$$
\zeta(s)(1-3^{1-s})=\sum_{m\geq1}\frac{P_{m}(s)\zeta(s+m)}{3^{s+m}}(1+2^m),
$$
which can be found in \cite[(2.16)]{LOC}.
Following the steps in \thmref{thm-trans-mobius}, we also prove the following theorem, which gives an extension
of Apostol's identity \eqref{TA-formula} for $\mathrm{Li}_z(s)$. 

\begin{thm}\label{TA-thm-2} 
Let $k>2$ be an integer and the other notations be as in Theorem \ref{trans-thm}. Then we have,
 \begin{equation}\label{TA-type-1}
 \sum_{d|k}\frac{\mu(d)}{d^s}\mathrm{Li}_{z^d}(s)=\sum_{ \substack{1\leq h\leq k/2 \\(h,k)=1}}z^hh^{-s}
 +\sum_{m\geq0}\frac{P_m(s)\mathrm{Li}_z(s+m)}{k^{s+m}}\sum_{\substack{1\leq h\leq k/2 \\(h,k)=1}}\left(z^{h}(-1)^m+z^{-h}\right)h^m.
 \end{equation}
 \end{thm}
 
Note that if we take $z=1$ in \eqref{TA-type-1}, we get back \eqref{TA-formula}, which in turn includes
\eqref{Ramaswami-3} and \eqref{Ramaswami-6}.  

%\begin{rmk}\label{rmk-trans-formula}\rm                                                                                  
%The polylogarithm function $\mathrm{Li}_z(s)$ satisfies the following translation formula if we take the sum for $1\leq h\leq k/2$ instead of $1\leq h\leq k/2$ with $(h,k)=1$ in \thmref{TA-thm-2} but now for $k\geq2$:
%\begin{align}\label{req-sum-2}
%\begin{split}
%&\mathrm{Li}_{z}(s)(1-1/k^s)+\frac{1+(-1)^k}{2k^s}z^{k/2}\phi(s,1/2,z)\\
%&=\sum_{1\leq h\leq k/2 }z^hh^{-s}+\sum_{m\geq0}\frac{P_m(s)\mathrm{Li}_z(s+m)}{k^{s+m}}\sum_{1\leq h\leq k/2}\left(z^h(-1)^m+z^{-h}\right)h^m.
%\end{split}
%\end{align}
%\end{rmk}
%
%We will give a sketch of the proof of \rmkref{rmk-trans-mobius} and \rmkref{rmk-trans-formula} after the proof of \thmref{thm-trans-mobius} and \thmref{TA-thm-2}.
%
%--------------------------------------------------------------------------------------------------------------------------------------------------------

\section{Proof of  Translation Formulae for $\mathrm{Li}_{z}(s)$}

Inspired by \cite{TA}, we begin with the Lerch zeta function 
$\phi(s,a,z)$. For $s\in \C$ with $\Re(s)>1$ and a real number $a>0$, $\phi(s,a,z)$ is defined by the following absolute convergent series,
\begin{equation}\label{Lerch}
\phi(s,a,z):=\sum_{n\ge 0}\frac{z^n}{(n+a)^s}.
\end{equation}
It is easy to see that $z \ \phi(s,1,z)=\mathrm{Li}_z(s)$.
%Note that if $z=1$ then $\phi(s,1,1)=\zeta(s)$.  So we take $z\neq 1$.
Note that it is easy to see that
the Lerch zeta function $\phi(s,a,z)$ satisfies the following indentities,
\begin{align}\label{trans-phi}
\phi(s,a,z)-z\phi(s,a+1,z)=a^{-s},
\end{align}
and 
\begin{align}\label{multi-formula}
\phi(s,ka,z)=k^{-s}\sum_{i=0}^{k-1}z^i\phi(s,i/k+a,z^k),
\end{align}
where $k\geq 1$ is an integer. The identity \eqref{multi-formula} can be seen as follows:
\begin{align*}
\sum_{n\ge 0}\frac{z^n}{(n+ka)^s}=k^{-s}\sum_{n\ge 0}\frac{z^n}{(n/k+a)^s}=k^{-s}\sum_{i=0}^{k-1}
\sum_{n\ge 0}z^i\frac{z^{kn}}{(n+i/k+a)^s}=k^{-s}\sum_{i=0}^{k-1}z^i\phi(s,i/k+a,z^k).
\end{align*}

Now for a fixed complex number $s$ with $\Re(s)>1$ and an integer $m\geq 0$, the $m$-th derivative of the Lerch zeta function
$\phi(s,a+1,z)$ with respect to $a$ is,
\begin{align*}
\diffp[m]{\phi(s,a+1,z)}{a}&=(-s)(-s-1)\cdots(-s-m+1){\phi(s+m,a+1,z)}\\
 &=(-1)^mm!P_m(s){\phi(s+m,a+1,z)}.
\end{align*}
Therefore, the Taylor series of ${\phi(s,a+1,z)}$ at $a=0$ is given by the series,
\begin{align}\label{pseries}
\begin{split}
\sum_{m\ge 0}(-a)^mP_m(s){\phi(s+m,1,z)}=\frac{1}{z}\sum_{m\ge 0}(-a)^mP_m(s)\mathrm{Li}_z(s+m).
\end{split}
\end{align}

We prove the following Lemma.
\begin{lem}\label{T-series}
For fixed complex numbers $s,z$ with $\Re(s)>1$ and $|z|=1$, the Taylor series \eqref{pseries} has the radius of convergence $1$.
Moreover, the Taylor series \eqref{pseries} converges uniformly to the Lerch zeta function $\phi(s,a+1,z)$ for $a \in (-1,1)$.
\end{lem}
\begin{proof}
Note that for a fixed complex number $s$ with $\Re(s)>1$, $\displaystyle\lim_{m\rightarrow \infty}\mathrm{Li}_z(s+m)=z$.
Also,
$$
\lim_{m\rightarrow \infty}\left\lvert\frac{P_{m+1}(s)}{P_{m}(s)}\right\rvert=\lim_{m\rightarrow \infty}\left\lvert\frac{s+m}{m+1}\right\rvert=1.
$$
Therefore, the radius of convergence of the series \eqref{pseries} is 1 and the series converges uniformly for $a \in (-1,1)$.
Next we want to show that for every $a \in (-1,1)$, the Taylor series in \eqref{pseries} converges to the function $\phi(s,a+1,z)$.
To do this, for each $n\geq 0$, we consider the remainder term $R_n(a)$, defined by
$$
R_n(a)=\phi(s,a+1,z)-\frac{1}{z}\sum_{m=0}^{n-1}(-a)^mP_m(s)\mathrm{Li}_z(s+m).
$$
So it is enough to show that for each $a$ with $|a|<1$, $\lim_{n\rightarrow\infty}R_n(a)=0$. Now by a
version of Taylor's theorem (see \cite[Corollary 31.6]{KAR}), for each $a\in(-1,1)$, there exists a real number $x$ in between 0 and $a$ such that 
\begin{align}\label{remainder}
\begin{split}
R_n(a)&=(-1)^nn!P_n(s){\phi(s+n,x+1,z)}\frac{(a-x)^{n-1}}{(n-1)!}a\\
&=(-1)^nnP_n(s){\phi(s+n,x+1,z)}(a-x)^{n-1}a.
\end{split}
\end{align}
Note that,
$$
|\phi(s+n,x+1,z)|\leq\frac{1}{(x+1)^{\Re(s)+n}}+\frac{1}{(2+x)^n}\sum_{m\ge 1}\frac{1}{(m+x+1)^{\Re(s)}}.
$$
Since $\Re(s)>1$ and $x$ is in between 0 and $a$, we have
$$
\lim_{n\rightarrow\infty}\frac{1}{(2+x)^n}\sum_{m\ge 1}\frac{1}{(m+x+1)^{\Re(s)}}=0.
$$
If $0<x<a<1$, then $\lim_{n\rightarrow\infty}\frac{1}{(x+1)^{\Re(s)+n}}=0$. Since the series \eqref{pseries} converges uniformly
for $a \in (-1,1)$, so does its formal derivative.
Therefore, $\lim_{n\rightarrow\infty}nP_n(s)a^n=0$ and hence $\lim_{n\rightarrow\infty}R_n(a)=0$, as $(a-x)^{n-1}a\leq a^n$.

Next let $-1<a<x<0.$ Note that it is enough to show that
$$
\lim_{n\rightarrow\infty}(-1)^nnP_n(s)\frac{(a-x)^{n-1}a}{(x+1)^{\Re(s)+n}}=0.
$$
For this, first note that $\lvert\frac{a-x}{x+1}\rvert\leq|a|$ for $-1<a<x<0$.  Therefore,
$$
\left| nP_n(s)\frac{(a-x)^{n-1}}{(x+1)^{n-1}}\frac{a}{(x+1)^{\Re(s)+1}}\right|\leq \frac{n|P_n(s)||a|^{n}}{(x+1)^{\Re(s)+1}}.
$$
Hence, by a similar argument as in the previous case, we have $\lim_{n\rightarrow\infty}R_n(a)=0$ when $-1<a<0$
This completes the proof of \lemref{T-series}.
\end{proof}

%Now we are ready to prove \thmref{trans-thm}.
\begin{proof}[Proof of \thmref{trans-thm}] Note that from \lemref{T-series} we have,
\begin{equation}\label{power-series}
z\phi(s,a+1,z)=\sum_{m\ge 0}(-a)^mP_m(s)\mathrm{Li}_z(s+m).
\end{equation}
Taking $a=-\frac{h}{k}$ in the above equation and summing for $0\leq h\leq k-1$ after multiplying by $z^{-h}$ on
both the sides of the above equation, we have,
\begin{align}\label{med-step-1}
\begin{split}
z\sum_{h=0}^{k-1}z^{-h}\phi\left(s,(k-h)/k,z\right)&=\sum_{m\ge 0}\frac{P_m(s)\mathrm{Li}_z(s+m)}{k^m}\sum_{h=0}^{k-1}z^{-h}h^m\\
&=\mathrm{Li}_z(s)\sum_{h=0}^{k-1}z^{-h}+\sum_{m\ge 1}\frac{P_m(s)\mathrm{Li}_z(s+m)}{k^m}\sum_{h=1}^{k-1}z^{-h}h^m.
\end{split}
\end{align}
Since $k=bq+1$, we have $\sum_{h=0}^{k-1}z^{-h}=k^{\delta_{q}}$. Hence the right-hand side of the above equation is
$$
k^{\delta_{q}}\mathrm{Li}_z(s)+\sum_{m\ge 1}\frac{P_m(s)\mathrm{Li}_z(s+m)}{k^m}\sum_{h=1}^{k-1}z^{-h}h^m.
$$
To calculate the left-hand side of \eqref{med-step-1}, we first substitute $a=1/k$ in \eqref{multi-formula}. Since $k=bq+1$
for some positive integer $b$ then \eqref{multi-formula} becomes
$$
k^s\phi(s,1,z)=\sum_{i=0}^{k-1}z^i\phi(s,(i+1)/k,z)=\sum_{i=0}^{k-1}z^{k-i-1}\phi(s,(k-i)/k,z)=\sum_{i=0}^{k-1}z^{-i}\phi(s,(k-i)/k,z).
$$
Since $z\phi(s,1,z)=\mathrm{Li}_z(s)$, using the above equations, we can write \eqref{med-step-1} as
\begin{align*}
\mathrm{Li}_{z}(s)(1-k^{\delta_{q}-s})=\sum_{m\ge 1}\frac{P_m(s)\mathrm{Li}_z(s+m)}{k^{s+m}}\sum_{h=1}^{k-1}z^{-h}h^m.
\end{align*}
This completes the proof.
\end{proof}

\begin{proof}[Proof of \corref{cor-Apostol-type-trans-1}]
 To prove \corref{cor-Apostol-type-trans-1}, we derive a closed formula for $\sum_{h=0}^{k-1}z^hh^m$, where $z$ 
 is a primitive $q$-th root of unity. The proof is then completed by the following lemma.
 \end{proof}
 \begin{lem}\label{sum}
Let $q\geq2$ be an integer and $z$ be a primitive $q$-th root of unity. Then for every positive integers $k,m$, we have
\begin{align}\label{m-power-sum-z}
\sum_{h=1}^{k-1}z^{h}h^m =\frac{1}{q}\sum_{j=1}^{q-1} \frac{z(z^{j}-1)}{z-1}\left(E_{q,m}(j)-z^{k-1}E_{q,m}(k+j-1)\right).
\end{align}
\end{lem}
\begin{proof}
 Note that the polynomial $E_{q,n}(x)$ satisfies the formula,
\begin{equation}\label{diff-sum}
E_{q,n}(x)+E_{q,n}(x+1)+\cdots+E_{q,n}(x+q-1)=qx^n, \text{\ \ \ for all $n\geq 0$ and $x\in \R$.}
\end{equation}
This follows from the fact that
$$
\sum_{n\ge 0}\left(E_{q,n}(x)+E_{q,n}(x+1)+\cdots+E_{q,n}(x+q-1)\right)\frac{t^n}{n!}=qe^{xt}.
$$
By comparing the coefficients of $t^n$ on both the sides of the above equation, we get the identity \eqref{diff-sum}. 
%Note that for $n\geq 1$ if we substitute $x=0$ in \eqref{diff-sum} then we get that
%$$
%E_{q,n}(0)+E_{q,n}(1)+\cdots+E_{q,n}(q-1)=0,
%$$
Therefore we have,
\begin{align*}
q\sum_{h=1}^{k-1}z^{h}h^m&=\sum_{h=1}^{k-1}z^h\sum_{i=0}^{q-1}E_{q,m}(h+i)=\sum_{j=1}^{k+q-2}E_{q,m}(j)
\sum_{\substack{h+i=j, \\1\leq h\leq k-1,\\ 0\leq i\leq q-1}}z^h\\                                         
&=\sum_{j=1}^{q-1}E_{q,m}(j)\sum_{h=1}^{j}z^h+\sum_{j=q}^{k-1}E_{q,m}(j)\sum_{h=j-q+1}^{j}z^h+ \sum_{j=k}^{k+q-2}E_{q,m}(j)\sum_{h=j-q+1}^{k-1}z^h.
\end{align*}
Since $1+z+\cdots+z^{q-1}=0$, the second sum on the right-hand side of the above equation is zero. Hence we can write,
\begin{align*}                                            
 q\sum_{h=1}^{k-1}z^{h}h^m&=\sum_{j=1}^{q-1}E_{q,m}(j)z\frac{z^{j}-1}{z-1}+\sum_{j=0}^{q-2}E_{q,m}(k+j)z^{k-q}\sum_{h=j+1}^{q-1}z^h\\
 &=\sum_{j=1}^{q-1}E_{q,m}(j)z\frac{z^{j}-1}{z-1}+\sum_{j=0}^{q-2}E_{q,m}(k+j)z^{k+j+1}\frac{z^{q-j-1}-1}{z-1}\\
 %&=\sum_{j=1}^{q-1}E_{q,m}(j)z\frac{z^{j}-1}{z-1}+\sum_{j=0}^{q-2}z^k\frac{1-z^{j+1}}{z-1}E_{q,m}(k+j)\\
 &=\sum_{j=1}^{q-1}E_{q,m}(j)z\frac{z^{j}-1}{z-1}+\sum_{j=1}^{q-1}z^k\frac{1-z^{j}}{z-1}E_{q,m}(k+j-1).
\end{align*}
This completes the proof of \lemref{sum}.
\end{proof}

%Note that if $z=-1$ then $q=2$, then $E_{2,m}(0)+E_{2,m}(1)=0$ for $m\geq 1$. Hence, by \lemref{sum},
%\begin{equation}\label{Euler-num}
%\sum_{h=0}^{k-1}(-1)^{h}h^m=\frac{1}{2}\left(E_{2,m}(0)+(-1)^{k-1}E_{2,m}(k)\right).
%\end{equation}
%This is a well-known identity, and can be found in \cite[24.4.8]{PEP}. 

\begin{proof}[Proof of \rmkref{rmk-trans-mobius}]
%Note that here we give only a sketch of the proof, as it is in the same spirit as the proof of \thmref{trans-thm}.
We substitute $a=h/k$ and multiply with $z^h$ on both the sides of \eqref{power-series}, where $1\leq h\leq k-1$.
Now we sum over $1\leq h\leq k-1$ and use \eqref{trans-phi}, followed by \eqref{multi-formula} (with $a=1/k$),
to obtain the required translation formula \eqref{rmk-trans-mobius-alter}.
 \end{proof}

\begin{proof}[Proof of \thmref{thm-trans-mobius}]
For integer $1\leq h\leq k$ such that $(h,k)=1$, substitute $a=h/k$ in \eqref{power-series}  and multiply both the sides by $z^h$. 
Sum for $1\leq h\leq k$ with $(h,k)=1$ and use \eqref{trans-phi} to get
\begin{align}\label{mobius-sum-1}
\sum_{ \substack{1\leq h\leq k \\(h,k)=1}}z^h\phi(s,h/k,z)=k^s\sum_{ \substack{1\leq h\leq k \\(h,k)=1}}z^hh^{-s}+
\sum_{m\geq0}(-1)^m\frac{P_m(s)\mathrm{Li}_z(s+m)}{k^m}\sum_{\substack{1\leq h\leq k \\(h,k)=1}}z^hh^m.
\end{align}
Now, the left-hand side of the above equation \eqref{mobius-sum-1} is
\begin{align*}
\sum_{ \substack{1\leq h\leq k \\(h,k)=1}}z^{h}\phi(s,h/k,z)
&=\sum_{h=1}^{k}\sum_{d|h, d|k}\mu(d)z^{h}\phi(s,h/k,z) =\sum_{d|k}\mu(d)\sum_{m=1}^{k/d}z^{md}\phi(s,md/k,z)\\
&=\sum_{d|k}\mu(d)\sum_{m=1}^{k/d}\sum_{n\ge 0}\frac{z^{n+md}}{(n+md/k)^s}
= k^s\sum_{d|k}\frac{\mu(d)}{d^s}\sum_{m=1}^{k/d}\sum_{n\ge 0}\frac{(z^d)^{\frac{kn}{d}+m}}{(\frac{kn}{d}+m)^s},    
\end{align*}
as $k=bq+1$. Thus we get
 \begin{align}\label{mobius-sum-2}
 \sum_{ \substack{1\leq h\leq k\\ (h,k)=1}}z^{h}\phi(s,h/k,z)=  k^s\sum_{d|k}\frac{\mu(d)}{d^s}\mathrm{Li}_{z^d}(s). 
\end{align} 
Therefore using \eqref{mobius-sum-2} in \eqref{mobius-sum-1}, and dividing both the sides by $k^s$, we get formula \eqref{trans-mobius-alter}.

If for integer $1\leq h\leq k$ such that $(h,k)=1$, we substitute $a=-h/k$ in \eqref{power-series} and multiply both the sides by $z^{-h}$ and then sum for
$1\leq h\leq k$ with $(h,k)=1$, we then get
\begin{align*}
\sum_{ \substack{1\leq h\leq k \\(h,k)=1}}z^{1-h}\phi(s,(k-h)/k,z)=\sum_{m\geq0}\frac{P_m(s)\mathrm{Li}_z(s+m)}{k^m}\sum_{\substack{1\leq h\leq k \\(h,k)=1}}z^{-h}h^m.
\end{align*}
Since $k=bq+1$ for some positive integer $b$, we can write the left-hand side of the above equation as
\begin{align*}
\sum_{ \substack{1\leq h\leq k \\(h,k)=1}}z^{k-h}\phi(s,(k-h)/k,z)=\sum_{ \substack{1\leq h\leq k \\(h,k)=1}}z^{h}\phi(s,h/k,z).
\end{align*}
Now using \eqref{mobius-sum-2}, we get the desired formula \eqref{trans-mobius}. This completes the proof.
\end{proof}

\begin{proof}[Proof of \thmref{TA-thm-2}]
Here we substitute $a=h/k$ with $1\leq h\leq k/2$ in \eqref{power-series}, multiply both the sides by $z^h$ and sum for $1\leq h\leq k/2$ with $(h,k)=1$ to get
\begin{align}
z\sum_{ \substack{1\leq h\leq k/2 \\(h,k)=1}}z^h\phi(s,(k+h)/k,z)=
\sum_{m\geq0}(-1)^m\frac{P_m(s)\mathrm{Li}_z(s+m)}{k^m}\sum_{\substack{1\leq h\leq k/2 \\(h,k)=1}}z^hh^m.
\end{align}
Using \eqref{trans-phi}  for $a=h/k$, we can write the above equation as
\begin{align}\label{sum1}
\sum_{ \substack{1\leq h\leq k/2 \\(h,k)=1}}z^h\phi(s,h/k,z)=k^s\sum_{ \substack{1\leq h\leq k/2 \\(h,k)=1}}z^hh^{-s}+
\sum_{m\geq0}(-1)^m\frac{P_m(s)\mathrm{Li}_z(s+m)}{k^m}\sum_{\substack{1\leq h\leq k/2 \\(h,k)=1}}z^hh^m.
\end{align}
Next substitute $a=-h/k$ in \eqref{power-series} with $1\leq h\leq k/2$ and sum for $1\leq h\leq k/2$ after multiplying by $z^{-h}$ to get,
\begin{align*}
\sum_{ \substack{1\leq h\leq k/2 \\(h,k)=1}}z^{1-h}\phi(s,(k-h)/k,z)=\sum_{m\geq0}\frac{P_m(s)\mathrm{Li}_z(s+m)}{k^m}\sum_{\substack{1\leq h\leq k/2 \\(h,k)=1}}z^{-h}h^m.
\end{align*}
Since $k=bq+1$ for some positive integer $b$, then we can write the left-hand side of the above equation as
\begin{align*}
\sum_{ \substack{1\leq h\leq k/2 \\(h,k)=1}}z^{k-h}\phi(s,(k-h)/k,z)=\sum_{ \substack{k/2\leq h\leq k \\(h,k)=1}}z^{h}\phi(s,h/k,z).
\end{align*}
Hence we get,
\begin{align}\label{sum2}
\sum_{ \substack{k/2\leq h\leq k \\(h,k)=1}}z^{h}\phi(s,h/k,z)=\sum_{m\geq0}\frac{P_m(s)\mathrm{Li}_z(s+m)}{k^m}\sum_{\substack{1\leq h\leq k/2 \\(h,k)=1}}z^{-h}h^m.
\end{align}
Since $k>2$, for $k$ even, $k/2$ is not coprime to $k$. So adding equation \eqref{sum1} and \eqref{sum2}, we have
\begin{align*}
\sum_{ \substack{1\leq h\leq k \\(h,k)=1}}z^{h}\phi(s,h/k,z)=k^s\sum_{ \substack{1\leq h\leq k/2 \\(h,k)=1}}z^hh^{-s}
+\sum_{m\geq0}\frac{P_m(s)\mathrm{Li}_z(s+m)}{k^m}\sum_{\substack{1\leq h\leq k/2 \\(h,k)=1}}\left(z^{h}(-1)^m+z^{-h}\right)h^m.
\end{align*}
Using \eqref{mobius-sum-2} and dividing by $k^s$, we get the desired formula \eqref{TA-type-1}.
This completes the proof.
\end{proof}

\begin{rmk}\label{rmk-trans-formula}\rm
If we forego the gcd condition in the sums of \eqref{TA-type-1}, we get the following variant of \eqref{TA-type-1}, applicable for $k \ge 2$:
%The polylogarithm function $\mathrm{Li}_z(s)$ satisfies the following translation formula if we take the sum for $1\leq h\leq k/2$ instead of $1\leq h\leq k/2$ with $(h,k)=1$ in \thmref{TA-thm-2} but now for $k\geq2$:
\begin{align}\label{req-sum-2}
\begin{split}
&\mathrm{Li}_{z}(s)(1-1/k^s)+\frac{1+(-1)^k}{2k^s}z^{k/2}\phi(s,1/2,z)\\
&=\sum_{1\leq h\leq k/2 }z^hh^{-s}+\sum_{m\geq0}\frac{P_m(s)\mathrm{Li}_z(s+m)}{k^{s+m}}\sum_{1\leq h\leq k/2}\left(z^h(-1)^m+z^{-h}\right)h^m.
\end{split}
\end{align}
\end{rmk}

\begin{proof}%[Proof of \rmkref{rmk-trans-formula}]
Following the steps as in the proof of \thmref{TA-thm-2}, but in this case summing over $1\leq h\leq k/2$ only with $k\geq2$, 
we obtain the following two identities:
\begin{align}\label{sum5}
\sum_{1\leq h\leq k/2 }z^h\phi(s,h/k,z)=k^s\sum_{1\leq h\leq k/2 }z^hh^{-s}
+\sum_{m\geq0}(-1)^m\frac{P_m(s)\mathrm{Li}_z(s+m)}{k^m}\sum_{1\leq h\leq k/2}z^hh^m,
\end{align}
\begin{align}\label{sum6}
\sum_{k/2\leq h\leq k-1 }z^h\phi(s,h/k,z)=\sum_{m\geq0}\frac{P_m(s)\mathrm{Li}_z(s+m)}{k^m}\sum_{1\leq h\leq k/2}z^{-h}h^m.
\end{align}
The left-hand sides of \eqref{sum5} and \eqref{sum6} add up to
$$
\sum_{h=1}^{k-1}z^h\phi(s,h/k,z) + \frac{1+(-1)^k}{2}z^{k/2}\phi(s,1/2,z).
$$
Now using \eqref{multi-formula} (with $a=1/k)$ and adding \eqref{sum5} and \eqref{sum6} we get,
\begin{align*}
&(k^s-1)\mathrm{Li}_z(s)+\frac{1+(-1)^k}{2}z^{k/2}\phi(s,1/2,z)\\
&=k^s\sum_{1\leq h\leq k/2 }z^hh^{-s}+\sum_{m\geq0}\frac{P_m(s)\mathrm{Li}_z(s+m)}{k^m}\sum_{1\leq h\leq k/2}\left(z^h(-1)^m+z^{-h}\right)h^m.
\end{align*}
Dividing by $k^{s}$ we get the required formula \eqref{req-sum-2}. 
This completes the proof of \rmkref{rmk-trans-formula}.
\end{proof}

Note that for $z=1$, \eqref{req-sum-2} recovers \cite[eq. (15)]{TA} (see \cite[eq. (2.15)]{LOC} for a corrected version).

\section{Some applications}
For complex number $s$ with $\Re(s)>1$, the Dirichlet series  $\mathrm{Li}_{-1}(s)=\sum_{n\ge 1}\frac{(-1)^{n}}{n^s}$ 
is related to $\zeta(s)$ as
 $$
 \mathrm{Li}_{-1}(s)=(2^{1-s}-1)\zeta(s).
 $$
Therefore, $\mathrm{Li}_{-1}(0)=B_1$ and for every $n\geq1$,
\begin{align}\label{special-values}
\mathrm{Li}_{-1}(-n)=(1-2^{n+1})\frac{B_{n+1}}{n+1}.
\end{align}
The above relation between $\zeta(s)$ and $\mathrm{Li}_{-1}(s)$ in fact allows us to derive many more
translation formulae satisfied by the Riemann zeta function.
Note that all the translation formulae in \S\ref{Sec:Trans-formula} also hold for meromorphic functions $\mathrm{Li}_z(s)$ in the whole complex plane.
We will use the translation formulae \eqref{trans-add} and \eqref{trans-sub} to find alternate expressions of the special values of $\mathrm{Li}_{-1}(s)$ at negative 
integers. Then we will apply \eqref{special-values} to get certain recurrence relations satisfied by the Bernoulli numbers $B_n$.

 For $s\in \C$ and integer $k\geq2$, define
 $$
 T(s,k):=\sum_{1\leq h\leq k-1}(-1)^hh^s.
 $$
For $z=-1$ (i.e. $q=2$ and $k>2$ odd), \eqref{trans-add} and \eqref{trans-sub} give us
\begin{align}\label{trans-1}
2 \ \mathrm{Li}_{-1}(s)(1-k^{-s})=T(-s,k)+2\sum_{m\ge1}\frac{P_{2m}(s)
\mathrm{Li}_{-1}(s+2m)}{k^{{s+2m}}}T(2m,k),
\end{align}
\begin{align}\label{trans-2}
T(-s,k)=2\sum_{m\ge1}\frac{P_{2m-1}(s)\mathrm{Li}_{-1}(s+2m-1)}{k^{{s+2m-1}}}T(2m-1,k).
\end{align}
%These translation formulae translate the function $\mathrm{Li}_{-1}(s)$ in a width 2 strip at once.

\begin{rmk}\rm
Note that $T(n,k)\in \Q$ for any integer $n$. If we take $k=3$ and $s=2$ in \eqref{trans-2} and use the fact that $\mathrm{Li}_{-1}(s)=(2^{1-s}-1)\zeta(s)$,
we get \eqref{zeta-series}.
%$$
%-\frac{3}{16}=\sum_{m\ge 1}\frac{m}{3^{2m+1}}\zeta(2m+1)(2^{2m-1}-1)(2^{-2m}-1).
%$$
\end{rmk}

Note that for any integers $m,n\geq 0$,
\begin{align*}
 P_m(-n)&=\left\{
    \begin{array}{ll}
    (-1)^m \binom{n}{m}  &  \text{ for } n \ge m,\\
    0 & \text{ for } n<m.
    \end{array}
\right.
\end{align*}
Now for $n\geq 1$, substituting $s=-2n$  in \eqref{trans-2} and $s=-2n+1$ in \eqref{trans-1}, we get
\begin{align}\label{special-trans-formula}
T(2n,k)+2\sum_{m=1}^{n}\binom{2n}{2m-1}\frac{\mathrm{Li}_{-1}(-2n+2m-1)}{k^{-2n+2m-1}}T(2m-1,k)=0,
\end{align}
and 
\begin{equation}\label{trans-formula-2}
\begin{split}
&2 \ \mathrm{Li}_{-1}(1-2n)(1-k^{2n-1})\\
&=T(2n-1,k)+2\sum_{m=1}^{n-1}\binom{2n-1}{2m}\frac{\mathrm{Li}_{-1}(1-2n+2m)}{k^{{1-2n+2m}}}T(2m,k),
\end{split}
\end{equation}
respectively. These lead us to the following formulae, on using \eqref{special-values}.
\begin{thm}\label{Bernoulli-recurrence-thm} 
For odd integer $k\geq 3$, the even indexed Bernoulli numbers satisfy the following relations:
%\begin{align*}
%B_0=1, \quad B_1=-\frac{1}{2}, \quad  B_{2n-1}=0 \text{ for all } n\geq 2
%\end{align*}
\begin{align}\label{Bernoulli-recurrence}
T(2n,k)=2\sum_{i=1}^{n}\binom{2n}{2i-1}\frac{B_{2i}}{2i}(2^{2i}-1)k^{2i-1}T(2n-2i+1,k),
\end{align}
and
\begin{equation}\label{Bernoulli-recurrence-2}
\begin{split}
&(1-2^{2n})(1-k^{2n-1})\frac{B_{2n}}{2n}\\
&=\frac{T({2n-1},k)}{2}+\sum_{i=1}^{n-1}\binom{2n-1}{2i-1}\frac{B_{2i}}{2i} k^{2i-1} (1-2^{2i}) T(2n-2i,k),
\end{split}
\end{equation}
for all $n\geq 1$.
\end{thm}

%\begin{proof}
%%Note that $B_0=\displaystyle \lim_{t\rightarrow0}\frac{t}{e^t-1}=1.$
%Using \eqref{special-values} in \eqref{special-trans-formula}, we get 
%\begin{align}\label{even-bernoulli-recurrence}
%T(2n,k)+2\sum_{m=1}^{n}\binom{2n}{2m-1}\frac{B_{2n-2m+2}}{2n-2m+2}(1-2^{2n-2m+2})k^{2n-2m+1}T(2m-1,k)=0.
%\end{align}
%%\begin{align}\label{odd-bernoulli-recurrence}
%%T(2n-1,k)+2\sum_{m=1}^{n}\binom{2n-1}{2m-1}\frac{B_{2n-2m+1}}{2n-2m+1}(2^{2n-2m+1}-1)k^{2n-2m}T(2m-1,k)=0
%%\end{align}
%%Since $T(1,k)=\frac{k-1}{2}\neq 0$ for all odd integers $k\geq 3$ then if we substitute $n=1$ in \eqref{odd-bernoulli-recurrence} 
%%we have $B_1=-\frac{1}{2}$. Now for 
%%$n\geq2$, we rewrite \eqref{odd-bernoulli-recurrence} as
%%\begin{align*}
%%&(2n-1)\frac{B_{2n-1}}{2n-1}(2^{2n-1}-1)k^{2n-2}T(1,k)\\
%%&+\sum_{m=2}^{n-1}\binom{2n-1}{2m-1}\frac{B_{2n-2m+1}}{2n-2m+1}(2^{2n-2m+1}-1)k^{2n-2m}T(2m-1,k)=0,
%%\end{align*}
%%and apply induction on $n$, to deduce that $B_{2n-1}=0$ for all $n\geq 2$.
%As $\binom{2n}{2m-1}=\binom{2n}{2n-2m+1}$, we can
%rewrite the equation \eqref{even-bernoulli-recurrence} as follows.
%$$
%T(2n,k)+\sum_{m=1}^{n}\binom{2n}{2n-2m+1}\frac{B_{2n-2m+2}}{2n-2m+2}(1-2^{2n-2m+2})k^{2n-2m+1}T(2m-1,k)=0.
%$$
%We replace $n-m+1$ by $i$ in the above equation to get the required recurrence \eqref{Bernoulli-recurrence}.
%This completes the proof of the \thmref{Bernoulli-recurrence-thm}.
%\end{proof}

As $T(s,3)=2^s-1$, taking $k=3$ we get the following corollary.
\begin{cor}\label{k=3}
For all $n\geq 1$, the even indexed Bernoulli numbers satisfy the following recurrence relations:
\begin{align}\label{Bernoulli-recurrence-3}
B_{2n}(2^{2n}-1)3^{2n-1}=\frac{2^{2n}-1}{2}+\sum_{i=1}^{n-1}\binom{2n}{2i-1}\frac{B_{2i}}{2i}(1-2^{2i})3^{2i-1}(2^{2n-2i+1}-1),
\end{align}
%Similarly if we substitute $s=-2n$ and $s=-2n+1$ for $n\geq1,$ in equation \eqref{trans-1}
%then we have the following theorem;
%\begin{thm}\label{Bernoulli-recurrence-thm-2}
%For odd integer $k\geq3$, the Bernoulli  numbers satisfy the following recurrence relation,
%$$
%B_0=1, \quad B_1=-\frac{1}{2}, \quad B_{2n-1}=0\text{ for all } n\geq 2, \quad
%$$
%and for all $n\geq1,$
%$$
%(1-2^{2n})(1-k^{2n-1})\frac{B_{2n}}{2n}=\frac{2^{2n-1}-1}{2}+\sum_{i=1}^{n-1}\binom{2n-1}{2i-1}\frac{B_{2i}}{2i} k^{2i-1} (1-2^{2i}) T(2n-2i,k).
%$$
%\end{thm}
%
%If we substitute $k=3$ in \thmref{Bernoulli-recurrence-thm-2} then we have the following recurrence relation,
\begin{align}\label{Bernoulli-recurrence-3}
(1-2^{2n})(1-3^{2n-1})\frac{B_{2n}}{2n}= \frac{2^{2n-1}-1}{2} +\sum_{i=1}^{n-1}\binom{2n-1}{2i-1}\frac{B_{2i}}{2i} \ 3^{2i-1}(1-2^{2i})(2^{2n-2i}-1).
\end{align}
\end{cor}

In fact, by taking different values of $k$, we can derive many such recurrence relations satisfied by
the even indexed Bernoulli numbers. The above recurrence relation \eqref{Bernoulli-recurrence-3} has an interesting application for the
tangent numbers. Recall that, for integer $n\geq 1$, $T_n$ denotes the $n$-th tangent number defined by
$$
T_n=(-1)^{n}2^{2n}(1-2^{2n})\frac{B_{2n}}{2n}.
$$
Note that $T_1=1$. Moreover, this is the only odd tangent number. This can be seen from the following corollary,
where we, more importantly, give an explicit congruence relation modulo $3$.

\begin{cor}\label{cor-tangent}
For all integers $n\geq 2$, the $n$-th tangent number $T_n$ is an even integer and not divisible by $3$. Moreover, for any $n \ge 1$,
 $$
  T_n \bmod{3}\equiv
  \begin{cases}
   1   &  \text{ if  } n  \text{ is odd},\\
    2  & \text{ if } n  \text{ is even}.
  \end{cases}
  $$
\end{cor}

\begin{proof} 
Put $t_n=(-1)^{n}T_n$. By multiplying both the sides of \eqref{Bernoulli-recurrence-3} by $2^{2n}$, we get
\begin{align}\label{2n-multiple}
(1-3^{2n-1})t_n= 2^{2n-1}(2^{2n-1}-1) +\sum_{i=1}^{n-1}\binom{2n-1}{2i-1}3^{2i-1}2^{2n-2i}(2^{2n-2i}-1)t_i.
\end{align}
As $t_i$'s are integers, note that for $n\geq 2$, each term in the right-hand side of the above equation is divisible by $4$.
However, as $1-3^{2n-1} \equiv 2 \bmod{4}$, we get 2 must divide $t_n$.

Note that $3$ divides all  the terms of the right-hand side of \eqref{2n-multiple}, except the first term. Also,
$$
2^{2n-1}(2^{2n-1}-1)  \equiv 2 \bmod{3}, \quad 1-3^{2n-1}  \equiv 1 \bmod{3}.
$$
Hence $t_n \equiv 2 \bmod{3}$, i.e.,  $T_n \equiv (-1)^n2 \bmod{3}$. This completes the proof of Corollary \ref{cor-tangent}.
\end{proof}

For $n \ge 2$, the tangent numbers $T_n$'s are not only even, but they also have a special feature modulo $10$.
This can be proved as a corollary of Theorem \ref{TA-thm-2}.

\begin{cor}\label{mod-5}
For $n\geq 2$, the last digit of tangent number $T_n$ is either $2$ or $6$, depending on $n$ is even or odd. More precisely,
for any $n \ge 1$,
$$
  T_n \bmod{5}\equiv
  \begin{cases}
   1   &  \text{ if  } n  \text{ is odd},\\
   2  & \text{ if } n  \text{ is even}.
  \end{cases}
  $$
\end{cor}

Taking $z=-1$ and $k=5$ in \eqref{TA-type-1}, we get the translation formula,
$$
\mathrm{Li}_{-1}(s)(1-5^{-s})=(2^{-s}-1)+2\sum_{m\geq 1}\frac{P_{2m}(s)\mathrm{Li}_{-1}(s+2m)}{5^{s+2m}}(2^{2m}-1).
$$
Now if we substitute $s=-2n+1$
%in the above equation then we get
%$$
%\mathrm{Li}_{-1}(-2n+1)(1-5^{2n-1})=(2^{2n-1}-1)+2\sum_{m=1}^{n-1}\frac{P_{2m}(-2n+1)\mathrm{Li}_{-1}(-2n+2m+1)}{5^{-2n+2m+1}}(2^{2m}-1).
%$$ 
for $n\geq 1$ and use \eqref{special-values}, we get the following identity:
$$
(1-5^{2n-1})(1-2^{2n})\frac{B_{2n}}{2n}=(2^{2n-1}-1)+2\sum_{m=1}^{n-1}\binom{2n-1}{2m}{5^{2n-2m-1}}(1-2^{2n-2m})\frac{B_{2n-2m}}{2n-2m}(2^{2m}-1).
$$
As $\binom{2n-1}{2m}=\binom{2n-1}{2n-2m-1}$, replacing $n-m$ by $i$, we get the following recurrence relation:
\begin{align*}
(1-5^{2n-1})(1-2^{2n})\frac{B_{2n}}{2n}=(2^{2n-1}-1)+2\sum_{i=1}^{n-1}\binom{2n-1}{2i-1}\frac{B_{2i}}{2i}{5^{2i-1}}(1-2^{2i})(2^{2n-2i}-1).
\end{align*}
Multiply both the sides of the above equation by $2^{2n}$ and put $t_n=(-1)^{n}T_n$ for all $n\geq 1$, we get the following recurrence for $t_n$:
\begin{align}\label{k=5}
(1-5^{2n-1})t_n=2^{2n}(2^{2n-1}-1)+2\sum_{i=1}^{n-1}\binom{2n-1}{2i-1}{5^{2i-1}}2^{2n-2i}(2^{2n-2i}-1)t_i.
\end{align}
We are now ready to prove Corollary \ref{mod-5}.

\begin{proof}[Proof of Corollary \ref{mod-5}]
For integer $n\geq 1$, using \eqref{k=5} we have $t_n \equiv 2^{2n}(2^{2n-1}-1) \bmod{5}$. It is easy to compute that
$$
  2^{2n}(2^{2n-1}-1) \bmod{5}\equiv
  \begin{cases}
   4   &  \text{ if  } n  \text{ is odd},\\
   2  & \text{ if } n  \text{ is even}.
  \end{cases}
  $$
Therefore,
$$
  T_n \equiv (-1)^n2^{2n}(2^{2n-1}-1) \bmod{5}\equiv
  \begin{cases}
   1   &  \text{ if  } n  \text{ is odd},\\
   2 & \text{ if } n  \text{ is even}.
  \end{cases}
  $$
Now we know from Corollary \ref{cor-tangent} that for $n \ge 2$, $T_n \equiv 0 \bmod{2}$. Hence by the Chinese remainder theorem
we get that
$$
  T_n \bmod{10}\equiv
  \begin{cases}
   6   &  \text{ if  } n  \text{ is odd},\\
   2  & \text{ if } n  \text{ is even}.
  \end{cases}
  $$
This completes the proof.
\end{proof}

\subsection*{Data availability statement}
This manuscript has no associated data.


\begin{thebibliography}{100}

\bibitem{TA}
T. M. Apostol, {\it Some series involving the Riemann zeta function}, Proc. Amer. Math. Soc.
{\bf 5} (1954), 239-243.

\bibitem{AIK}
T. Arakawa, T. Ibukiyama and M. Kaneko, {Bernoulli Numbers and Zeta Functions}, Springer Monographs in Mathematics, {\it Springer, Tokyo}, 2014.

\bibitem{BBC}
J. M. Borwein, D. M. Bradley and R. E. Crandall,
{\it Computational strategies for the Riemann zeta function}, J. Comput. Appl. Math. {\bf 121} (2000), no. 1-2, 247--296.

\bibitem{HL}
G-N. Han and  J-Y. Liu,
{\it Combinatorial proofs of some properties of tangent and Genocchi numbers},
European J. Combin. {\bf 71} (2018), 99--110.

\bibitem{KB}
D.E. Knuth and T.J. Buckholtz, 
{\it Computation of tangent, Euler, and Bernoulli numbers},
Math. Comp. {\bf 21} (1967), 663--688.


\bibitem{LOC}
H. Lee, B. M. Ok, J. Choi, {\it Notes on some identities involving the Riemann zeta function}, Commun. Korean Math. Soc. {\bf 17} (2002), no. 1, pp. 165--173.

\bibitem{PB}
P. S. Mehta, B. Saha, {\it Multiple polylogarithms, a regularisation process and an admissible open domain of convergence} (communicated),
\href{https://arxiv.org/pdf/2511.00889}{https://arxiv.org/pdf/2511.00889}.

%\bibitem{PEP}
%NIST, Digital Library of Mathematical Functions, \url{https://dlmf.nist.gov/}.

\bibitem{VR}
V. Ramaswami, {\it Notes on Riemann's $\zeta$-function}, J. London Math. Soc. {\bf 9} (1934), 165--169.

\bibitem{KAR}
K. A. Ross, Elementary analysis, The theory of calculus, 2nd edition, Undergrad. Texts Math., {\it Springer, New York}, 2013.

\bibitem{BS}
B. Saha, An elementary approach to the meromorphic continuation of some classical Dirichlet series,
Proc. Indian Acad. Sci. Math. Sci. {\bf 127} (2017), no. 2, 225--233.

\bibitem{ZLZ}
T. Zhao, Z. Lin and Y. Zang,
{\it An identity relating Catalan numbers to tangent numbers with arithmetic applications},
European J. Combin. {\bf 132} (2026), part B, Paper No. 104283.


%\bibitem{SC}
%H. M. Srivastava and J. Choi, Zeta and $q$-Zeta functions and associated series and integrals, {\it Elsevier, Inc., Amsterdam}, 2012.

\end{thebibliography}
\end{document}